\documentclass[11pt]{amsart}

\usepackage{tasks}
\usepackage{hyperref}
\usepackage{blindtext}
\usepackage{setspace, microtype}
\usepackage{marginnote}
\usepackage{datetime}
\usepackage{color}
\usepackage{amsrefs}
\usepackage{amsmath,amssymb,latexsym,cases}
\usepackage{mathtools}
\usepackage{graphicx}  
\usepackage[top=3.5cm, bottom=2.5cm, outer=2.5cm, inner=2.5cm]{geometry}
\usepackage{hyperref}
\hypersetup{linkcolor=red,
citecolor=black
}
\usepackage{enumerate}   

\usepackage{lineno}

\newtheorem{thm}{Theorem}[section]
\newtheorem{lem}[thm]{Lemma}
\newtheorem{prop}[thm]{Proposition}

\newtheorem{lemma}[thm]{Lemma}
\theoremstyle{definition}

\theoremstyle{remark}
\newtheorem{re}[thm]{Remark}

\numberwithin{equation}{section}

\usepackage[utf8]{inputenc}
\usepackage[english]{babel}
\usepackage{paralist}
\usepackage{amsthm}
 
\theoremstyle{definition}
\newtheorem{definition}{Definition}[section]
 
\theoremstyle{remark}

 \usepackage{bm,amsmath}
 
\newcounter{mnote}
\newcommand{\tran}{^{\mathstrut\scriptscriptstyle\top}} 
 
\newcommand{\RN}[1]{%
  \textup{\uppercase\expandafter{\romannumeral#1}}%
}

\newcommand{\Rey}{\text{Re}}
\newcommand{\X}{\mathbb{X}_{t}}
\newcommand{\XX}{\mathbb{X}} 
\newcommand{\E}{\mathbb{E}}
\newcommand{\grad}{\nabla}

\newcommand{\tU}{\widetilde U}
\newcommand{\tRe}{\widetilde{\text{Re}}}
\allowdisplaybreaks[4]
\begin{document}

\title{Three-dimensional shear driven turbulence with noise at the boundary}

\author{Wai-Tong Louis Fan}
\address{Department of Mathematics, Indiana University Bloomington, IN 47405, USA}
\email{waifan@iu.edu}
\author{Michael Jolly}
\address{Department of Mathematics, Indiana University Bloomington, IN 47405, USA}
\email{msjolly@indiana.edu}
\author{Ali Pakzad}
\address{Department of Mathematics, Indiana University Bloomington, IN 47405, USA}
\email{apakzad@iu.edu}

\maketitle
\setcounter{tocdepth}{1}

\begin{abstract}
We consider the incompressible 3D Navier-Stokes equations subject to a shear induced by noisy movement of part of the boundary. The effect of the noise is quantified by upper bounds on the first two moments 
of the dissipation rate. 
The expected value estimate is consistent 
with the Kolmogorov dissipation law, recovering an upper bound as in \cite{DC92} for the deterministic case. The movement of the boundary is given by an Ornstein--Uhlenbeck process;
a potential for over-dissipation is noted if the Ornstein--Uhlenbeck process were replaced by the Wiener process.
\end{abstract}

\section{INTRODUCTION}

Noise is added to turbulence models for a variety of reasons, both practical and theoretical. For example, the onset of turbulence is often related to the randomness of background movement \cite{MR04}. In any turbulent flow there are unavoidably perturbations in boundary conditions and material properties; see \cite[Chapter 3]{P00}. The addition of noise in a physical model can be interpreted as a perturbation from the model. There is considerable evidence supporting the stabilization of solutions by noise (see, e.g., \cite{B01,CLM01,FSQD19,K99}). However, the effect of noise in turbulent flow is far from completely understood.

This paper concerns the Kolmogorov dissipation law associated with the  incompressible  Navier-Stokes equations (NSE) in a 3-dimensional box $D = (0,L)^2 \times (0,h)$ subject to a shear induced by noisy movement of one wall.  Specifically, we consider the following 
differential equation,
\begin{equation} \label{SNSE}
\begin{split}
 d u +  ( u \cdot \nabla u - & \nu \Delta u + \nabla p  ) \, dt= 0,\\
& \grad \cdot u  =0   ,
\end{split}
\end{equation}
with $L$-periodic boundary condition in the $x_1$ and $x_2$ directions and a random boundary condition  given by the following:  for all time $t\in \mathbb{R}_+$ and $(x_1,x_2)\in (0,L)^2$,
\begin{equation} \label{BC}
\begin{split}
   u(x_1,x_2,0,t)=(\X,0,0)\tran \hspace{18pt} &\mbox{and}\hspace{18pt} u(x_1,x_2,h,t)=(0,0,0)\tran\;.
\end{split}
\end{equation}
In the above, $\nu>0$ is a fixed real parameter representing the viscosity, and
$\XX=(\X)_{t\in \mathbb{R}_+}$ is a given continuous-time, real-valued stochastic process.
The stochastic processes  
$u$ and $p$
represent respectively the velocity field and the pressure.


The Kolmogorov dissipation law is tied to a phenomenon in turbulence called the energy cascade, which can be explained in 3 main steps. $1-$ In the absence of a body force, the kinetic energy is introduced into the large scales of the fluid between the parallel plates by the effects of the moving plate. This energy is called \textit{energy input}. $2-$  The large eddies break up into smaller eddies through vortex stretching over an \textit{intermediate range}, where the energy is transferred to smaller scales and the energy dissipation due to the viscous force is negligible. $3-$ At small enough scales (expected to be $ \sim \Rey^{-3/4}$, where $\Rey$ is the Reynolds number defined in \eqref{upbnd}) \textit{dissipation dominates} and the energy in those smallest scales decays to zero exponentially fast. 

Based on the  above description the  dissipation is effective  at the end of a sequence of processes. Therefore, the rate of dissipation, which measures the amount of energy lost by the viscous force, is determined by the first process in the sequence, which
is the energy input.  The persistent force driving the shear flow is the motion of the bottom wall $\{(x_1,x_2,0):\,(x_1,x_2)\in [0,L]^2\}$.
The time averaged energy dissipation rate must balance the drag exerted by the walls on the fluid. In terms of the characteristic speed $U$, the large eddies have energy of order $U^2$ and time scale $\tau=h/U$, so the rate of energy input can be scaled as $U^2/\tau =U^3/h$. This suggests the Kolmogorov dissipation law for time-averaged energy dissipation rate $ \varepsilon$ (Kolmogorov 1941);
$$  \varepsilon  \sim \frac{U^3}{h}.$$
Here $a \sim b$ means  $a\lesssim b$ and $b\lesssim a$;  $a\lesssim b$ means $a\le c\,b$ for a nondimensional universal constant $c$.

The energy dissipation rate  has been widely studied  in the literature in the deterministic case \cite{B70,DLPRSZ18,DF02,DR00,H72,L02,L16,AP17,AP19,AP20}.   Doering and Constantin proved in \cite{DC92} a rigorous asymptotic bound  directly from the Navier-Stokes equations. Their bound is of the form 
\begin{align}\label{upbnd}  \varepsilon  \lesssim \frac{U^3}{h}\;, \quad  \text{as} \ \Rey \to \infty, \quad \text{where} \quad  \Rey= \frac{Uh}{\nu},\;
\end{align}
similar estimations have been proven by Kerswell \cite{K97},  Marchiano \cite{M94}, and Wang \cite{W97} in more generality. 

In this paper we 
choose $\X$ to be an Ornstein--Uhlenbeck process (OU process) satisfying \eqref{OU}. We
derive an upper bound on the expected value of the energy dissipation rate as well as its second moment in terms of characteristics of the randomly moving  bottom wall. 
Our estimate recovers \eqref{upbnd} in the limit as the variance $\sigma^2$ of the noise tends to 0.  The key to the analysis is the choice of a stochastic background flow and the treatment of a stochastic integral (with respect to the Wiener process) as a local martingale.

Since the work of Bensoussan and Temam \cite{BT73} in 1973, there has been substantial advance in understanding the stochastic Navier-Stokes equations,  see for example \cite{BCPW19,B00,BP00,MR04,MR05,WW15} and the references therein. Recently in \cite{CP20}, the exact dissipation rate is obtained for the stochastically forced Navier-Stokes equations under an assumption of energy balance.  In all those works the equation always contains noise as a forcing term. Other than the analysis of symmetries of a passive scalar advected by a shear flow in which a boundary moves as a stochastic process in \cite{CamassaKilicMcLaughlin19}, to the best of our knowledge,  there is no other work concerning  the equations of the motion with stochastic boundary conditions. 


\subsection*{Organization of this paper} In section 2, we will introduce the necessary notation and preliminary results needed in the proceeding sections. In section 3, we will state the main result of this work. We will set up an almost sure bound starting from the energy equation in section 4. From there, we will derive an upper bound on the mean value and variance of the energy dissipation respectively in sections 5 and 6.  The concluding Section 7 contains some open problems in this direction.

\section{Definitions and  Notations}

In this paper, we choose $\X$ to be an \textbf{\textit{Ornstein--Uhlenbeck process}} (OU process), which is a diffusion process solving the It\^o stochastic differential equation
\begin{equation}\label{OU}
d\X = \theta (U - \X) dt + \sigma d W_t,
\end{equation}
where $W=(W_t)_{t\in\mathbb{R}_+}$ is 
a standard Brownian motion (a.k.a. the Wiener process), and $\theta >0$ and $\sigma >0$ are parameters. A strong solution to \eqref{OU} is given by
\begin{align*}
\X  = \mathbb{X}_0\,e^{-\theta t} + U\,(1-e^{-\theta t}) + \sigma \int_0^t e^{-\theta (t-s)}\, dW_s.
\end{align*}
It is well known that $\X$ has stationary distribution given by the normal distribution $\mathcal{N}(U, \frac{\sigma^2}{2 \theta})$ with mean $U$ and variance $\frac{\sigma^2}{2\theta}$.  If the initial distribution satisfies $\XX_0  \sim \mathcal{N}(U, \frac{\sigma^2}{2 \theta})$, then $\X \sim N(U, \frac{\sigma^2}{2 \theta})$ for all $t\geq 0$ and we say $\XX$ is a {\it stationary OU} process.


Intuitively,  the OU process is a  Wiener process plus a tendency to move towards a location $U$, where the tendency is greater when the process is further away from that location. 
In (\ref{OU}), $\theta$ is the decay-rate which measures how strongly the system reacts to perturbations, and  $\sigma ^2$ is the variation or the size of the noise. 
We will need the following basic properties of the stationary OU process
(for a proof and additional properties see \cite{Doob42}).


\begin{prop}\label{OUProperties}
Let $\X$ be a stationary Ornstein--Uhlenbeck process satisfying \eqref{OU}. The following hold for all $t\geq 0$.
\begin{enumerate}[(i)]
\item  $\X \sim N(U, \frac{\sigma^2}{2 \theta})$\;,
\item  $[\XX]_t = \sigma^2\,t$, where $[\XX]_t$ is the quadratic variation of $\XX$ on $[0,t]$.
\end{enumerate}
\end{prop}

Throughout this manuscript, the $L^2(D)$  norm and inner product will be denoted by $\|\cdot\|$ and $( \cdot ,  \cdot)$  respectively.
For the sake of boundary conditions, we consider 
\begin{align*}
    H=\{v\in [L^2(D)]^3: \,   \nabla  \cdot v =0,\,  v(x_1,x_2,0)=v(x_1,x_2,h)=0, \,  v \text{ periodic in } x_1,x_2 \},\\
V=\{v\in [H^1(D)]^3: \, \nabla \cdot v =0, \,  v(x_1,x_2,0)=v(x_1,x_2,h)=0, \,  v \text{ periodic in } x_1,x_2 \},\\
C^{\infty}_{\text{div}}=\{v\in [C^{\infty}(D)]^3: \, \nabla \cdot v =0, \, v(x_1,x_2,0)=v(x_1,x_2,h)=0, \,  v \text{ periodic in } x_1,x_2 \}.
\end{align*}





\medskip

\subsection*{Stochastic Background Flow}
The difficulty in the analysis of the shear flow (\ref{BC}) is due to the effect of the random inhomogeneous boundary condition. We overcome this difficulty by constructing a carefully chosen stochastic background flow. This construction is based on the Hopf extension \cite{H55}.

Our key idea here is to choose the boundary layer thickness $\delta_t$ in the background flow to be random and time-dependent, namely,
\begin{equation}\label{delta}
\delta_t = \delta(\X(\omega)) 
= \frac{A}{|\X(\omega)|^2 + B}
\end{equation}
where $\delta:\,\mathbb{R}\to (0,\infty)$ is the function $\delta(z)=\frac{A}{z^2+B}$.
We later choose $A=\nu U$ and $B=U^2$, so $\delta_t$ has the dimension of length and
$\delta_t\in (0,h)$ if $\Rey = \frac{U\, h}{\nu} > 1$; see Lemma \ref{L:IneqForDelta} for  precise requirements.


We then let $\phi:\,[0,h]\times\mathbb{R}\to \mathbb{R}$ be the function 
\[
\phi(a,z)= 	\left(1- \dfrac{a}{\delta(z)}\right)z\,1_{\{ 0\leq a\leq \delta(z) \}}.
\]
By definition, we have (see Figure \ref{fig:phi})
\begin{equation} \label{Background}
\phi(x_3,\X(\omega)) =
\left\{
	\begin{array}{ll}
		\left(1- \dfrac{x_3}{\delta_t}\right) \X(\omega)  & \mbox{if } 0\leq x_3\leq \delta_t  \\
		0 &  \mbox{if } \delta_t\leq x_3\leq h
	\end{array}.
\right.
\end{equation}
\begin{figure}[t!]
    \centering
    \includegraphics[scale=0.3]{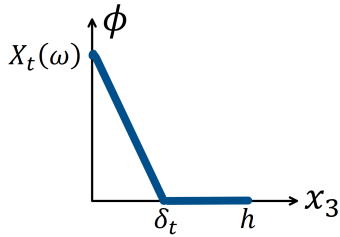}
    \caption{The graph of $x_3 \mapsto \phi(x_3,\X(\omega))$, where $\delta_t=\delta(\X(\omega))$ is the boundary layer thickness.}
    \label{fig:phi}
\end{figure}
Finally, we define the \textbf{stochastic background flow} $\Phi=\Phi_t(x_1,x_2,x_3;\omega)$ as 
\begin{equation}\label{Background-vec}
\Phi_t(x_1,x_2,x_3;\omega) := \big(\phi(x_3,\X(\omega)) ,\;0,\;0\big)\tran.
\end{equation}
There can be other choices for the function $\delta_t$, and our choice in (\ref{delta}) is motivated by the general analysis  in \eqref{Eq2_v2}.  The \textbf{boundary layer} is denoted by $D_{\delta} = (0, L)^2 \times (0, \delta_t)$.


\medskip

\subsection*{Martingale solutions}
We follow the standard notion of martingale solutions for stochastic Naiver Stokes equations such as Flandoli and Gatarek \cite[Definition 3.1]{FG95}, 
and define a martingale solution for our system \eqref{SNSE}-\eqref{BC}.
This notion is a  probabilistically weak analogue of the Leray-Hopf weak solution to the deterministic Navier–Stokes equations.

\begin{definition}[Martingale solution on compact intervals]
\label{MgaleCmpct}
Let $T\in [0,\infty)$. A martingale solution to \eqref{SNSE}-\eqref{BC}
on $[0, T ]$ consists of a stochastic basis $\big(\Omega,\, (\mathcal{F}_t)_{t\in[0,T]}, \,\mathbb{P}\big)$
with a complete right-continuous filtration $(\mathcal{F}_t)_{t\in[0,T]}$, a stationary OU process $(\X)_{t\in [0,T]}$ adapted to $(\mathcal{F}_t)_{t\in[0,T]}$, and with mean $U$ and variance $\frac{\sigma^2}{2\theta}$,
and an $\mathcal{F}_t$-progressively measurable stochastic process
\[
u:\; [0,T] \times \Omega  \rightarrow  [L^2(D)]^3
\]
such that 
\begin{itemize}

\item 
 $u-\Phi$  has sample paths in $  L^{2}\left([0,T];\,V\right) \cap L^{\infty}\left([0,T];\,H\right)$  almost surely, 
    \item for all $t\in [0,T]$ and all   $\varphi \in C^{\infty}_{\text{div}}$, the following identity holds almost surely,
    \begin{equation}\label{ByParts}
        (u (t) , \varphi) + \nu \int_0^t (\nabla u(s) , \nabla \varphi) \, ds + \int_0^t (u(s) \cdot \nabla u(s) , \varphi) \, ds = (u(0), \varphi),
    \end{equation}

 \item the following holds
 \begin{equation}\label{Condition1u}
 \mathbb{E} \left[  \sup_{s \in [0,T]} \|u (s)\|^2 + \int_0^T  \|\nabla u(s)\|^2\, dt \right]\, < \infty.
 \end{equation}
\end{itemize}
\end{definition}

\begin{re}
The existence of a martingale  solution under the current assumptions can be derived by modifying a classical result of Flandoli and Gatarek \cite{FG95} in the case when $\X = 0$. As in the deterministic case, the uniqueness of
such solutions is an open problem.  
\end{re}

\begin{re}
Note that above solution is independent of the choice of $\Phi$  and depends only on the value of $\Phi$ on the boundary; see for instance \cite[Chapter 9]{BH10}.
\end{re}

Essentially, a global solution has a fixed stochastic basis over $[0,\infty)$
which, when restricted to $[0,T]$, yields a solution as in Definition \ref{MgaleCmpct}.

\begin{definition}[Martingale solution]\label{Def:MtgSol}
A martingale solution to \eqref{SNSE}-\eqref{BC}  consists of a stochastic basis $\big(\Omega,\, (\mathcal{F}_t)_{t\in[0,\infty)}, \,\mathbb{P}\big)$
with a complete right-continuous filtration $(\mathcal{F}_t)_{t\in \mathbb{R}_+}$,  a stationary OU process $\XX=(\X)_{t\in \mathbb{R}_+}$ with mean $U$ and variance $\frac{\sigma^2}{2\theta}$,
and an $\mathcal{F}_t$-progressively measurable stochastic process
\[
u \in [0,\infty) \times \Omega \rightarrow [L^2(D)]^3
\]
such that $\left\{
\big(\Omega,\, (\mathcal{F}_t)_{t\in[0,T]}, \,\mathbb{P}\big),\;(\X)_{t\in [0,T]},\;
u\big|_{[0,T]\times\Omega} \right\}$ is a martingale solution to \eqref{SNSE}-\eqref{BC}  on $[0,T]$ for all $T\in [0,\infty)$.
\end{definition}


\medskip

\subsection*{Energy dissipation rate }
In experiments, it is natural to take a long, but fixed time interval $[0,T]$ and compute the time-average 
\begin{equation}\label{Def:timeaverageT}
\langle \epsilon\rangle_T:=\frac{1}{|D|}\,\frac{1}{T}\int_{0}^{T}   \nu\|\nabla u(t ,\cdot, \omega)\|_{L^2}^2  \, dt \;.
\end{equation}
It is shown in \cite{FJMRT} that the effect of $T$ in finite-time averages of physical quantities in turbulence theory, including the energy dissipation rate, can be controlled by parameters such as $\Rey$.   In our setting, this finite-time average in \eqref{Def:timeaverageT}  is a random variable whose mathematical expectation can be approximated by taking an average over a number of samples in the experiments. 
\begin{definition}\label{Def:dissipation}
We take the  time-averaged expected energy dissipation rate for a martingale solution $u$ of \eqref{SNSE}-\eqref{BC} to be defined by
\begin{align}\label{varepsdef}
\varepsilon  \coloneqq \limsup\limits_{T\rightarrow\infty}\mathbb{E}[\langle \epsilon\rangle_T]
=\limsup\limits_{T\rightarrow\infty}\mathbb{E}\left[ \frac{1}{|D|}\,\frac{1}{T}\int_{0}^{T}   \nu\|\nabla u(t ,\cdot, \omega)\|_{L^2}^2  \, dt\right]\;.
\end{align}
\end{definition}
Our main result, Theorem \ref{T:main} below, is an upper bound for $\varepsilon$ in terms of the characteristics of the noise added to the movement of the boundary. 
The variance $Var[\langle \epsilon\rangle_T]$ is bounded by the second moment
$\mathbb{E}[\langle \epsilon\rangle_T^2]$. In this work, we obtain an upper  bound for the limsup of $\mathbb{E}[\langle \epsilon\rangle_T^2]$.  Our method can readily be generalized to give an upper bound for the $p$-th moment for all $p\geq 1$; see Remark \ref{Rk:higher}.


\begin{re}
We note that by
Fatou's lemma
$$\limsup\limits_{T\rightarrow\infty}\mathbb{E}[\langle \epsilon\rangle_T] \leq \mathbb{E}\left[\limsup\limits_{T\rightarrow\infty}\langle \epsilon\rangle_T \right].  $$
Hence our upper bound on $\varepsilon$ defined in  \eqref{varepsdef} does not imply one when the order of the lim sup and expectation are reversed.
\end{re}

\section{Statement of the Results}
\begin{thm}\label{T:main}
Suppose $\left\{
\big(\Omega,\, (\mathcal{F}_t)_{t\in[0,\infty)}, \,\mathbb{P}\big),\;\XX,\;
u \right\}$ is a martingale solution to \eqref{SNSE}-\eqref{BC}, where $\XX$ is a stationary Ornstein--Uhlenbeck process \eqref{OU}. Assume that $\Rey = \frac{U\, h}{\nu} > 1$ and that the initial condition $u(0)$  is such that $\mathbb{E} [\|u(0)\|^2] <\infty$.  
Then the energy dissipation rate \eqref{varepsdef} satisfies
\begin{align}\label{expvalbound}
\varepsilon =\limsup\limits_{T\rightarrow\infty}\mathbb{E}[\langle \epsilon\rangle_T]  
\leq 32 \,  \frac{U^3}{h} + 2\left( 6  \, \frac{1}{\Rey} + 28 \, \frac{U}{ h \, \theta} + 12\, \frac{1}{\Rey^2} \,  \frac{h \, \theta}{U} +  24 \frac{1}{\Rey^2} \, \frac{h \, \sigma^2}{U^3} + 6\, \frac{\sigma^2}{h \, U \, \theta^2}\right)\, \sigma^2.
\end{align}
Moreover, the second moment  of $\langle \epsilon\rangle_T$ satisfies
\begin{equation}
\begin{aligned}\label{variancebound}
\limsup\limits_{T\rightarrow\infty}\mathbb{E}[\langle \epsilon\rangle_T^{2}]   
\lesssim
\frac{U^{6}}{h^2} + \sigma^2 \,P(\sigma)
\end{aligned}
\end{equation}
where $P(\sigma)=P_{U,\nu,\theta}(\sigma)$ is an explicit polynomial in $\sigma$ whose coefficients are explicit functions of $U,\,\nu$ and $\theta$.
\end{thm}

\medskip

In the above estimate on the mean of the dissipation rate (\ref{expvalbound}), as the variance $\sigma$ of the disturbance from $U$ tends to $0$, we recover the upper bound in Kolmogorov's dissipation law, 
 $$ \lim_{\sigma \rightarrow 0}\,  \varepsilon  \lesssim  \frac{U^3}{h},$$
which is also consistent with the rate proven for the Navier-Stokes equations in \cite{DC92}.  The constants suppressed by the use of $\lesssim$ in \eqref{variancebound}
is explicitly given in  \eqref{varianceboundexplicit} for the second moment.
\begin{re}
Since  $U$ is the mean velocity of the bottom wall, $\X$ has the dimension of velocity.  Therefore, $\theta$ scales as  $\frac{1}{\text{time}}$, and  $\sigma$ has dimension $\frac{\text{velocity}}{\sqrt{\text{time}}}$.  Therefore, one can check that the results in Theorem \ref{T:main} are also dimensionally consistent. 
\end{re}

\section{An almost sure bound on the energy dissipation}

In this section, we prove an almost sure upper bound for the energy dissipation.
We will see that $\delta_t$ in (\ref{delta}) is determined so as to absorb a term involving $\|\nabla v\|$ in \eqref{Eq2_v2}.

With $\delta(z)= \frac{A}{z^2 + B}$, we
let  $\phi(x_3,\X(\omega))=f(\X(\omega))$ where $f:\mathbb{R}\to \mathbb{R}$ is the smooth function,  
\begin{equation}\label{Def:f}
    f(z)=f_{x_3}(z)=\left(1- \dfrac{x_3}{\delta(z)}\right)z\;,
\quad\text{for}\quad x_3\in (0,\,\delta_t)\;.\end{equation} 
It\^o's rule asserts that $\mathbb{P}$-a.s. we have 
\begin{equation}\label{Ito1}
\begin{split}
df(\X) =& f'(\X)\,d\X \,+\,\frac{\sigma^2}{2}f''(\X)\,dt \\
=& f'(\X)\,[\theta (U - \X) dt + \sigma d W_t] \,+\,\frac{\sigma^2}{2}f''(\X)\,dt \\
=& \mathcal{L}f(\X)\,dt \,+\,
 \sigma f'(\X) d W_t 
\end{split}
\end{equation}
for $t\geq 0$, where we used the equation \eqref{OU} of the OU process in the second equality, and 
\begin{equation}\label{Def:Generator}
\mathcal{L}f(z)=f'(z)\,\theta (U - z) + \frac{\sigma^2}{2}f''(z).
\end{equation}
We can extend $\mathcal{L}$ to a differential operator
which is the infinitesimal generator of the OU process.

Before proceedung to the main analysis, we gather some basic calculations  in Lemma \ref{L:Derivatives} below. First note that \eqref{Def:f} implies that
\begin{align}
f'(z)=&\, 1-x_3 \frac{\delta -z\delta'}{\delta^2} \notag\\
f''(z)=&\,
x_3 \frac{z\delta^2\delta'' +2\delta^2\delta'  -2z\delta(\delta')^2}{\delta^4}.\label{Df_0}
\end{align}

\begin{lem}\label{L:Derivatives}
Consider $\delta(z)= \frac{A}{z^2 + B}$ and     $f(z)=f_{x_3}(z)=\left(1- \dfrac{x_3}{\delta(z)}\right)z$ as above. Then, 
\begin{equation}\label{Ddelta}
\delta'(z)=\frac{-2Az}{(z^2 + B)^2} \quad \text{and} \quad
\delta''(z)=\frac{2A(3z^2-B)}{(z^2 + B)^3}.
\end{equation}
Hence from \eqref{Df_0} we have
\begin{equation}\label{Df}
f'(z)=1-x_3\frac{3z^2+B}{A} \quad\text{and}\quad f''(z)=
-x_3 \,\frac{6z}{A}.
\end{equation}
\end{lem}

A basic tool in  the mathematical understanding of the dissipation rate is the energy inequality, which is obtained \textit{formally} by taking the scalar product of the equations by a solution.  However  in the case of shear flow here, the viscosity term cannot be handled by integration by parts  due to the effect of the inhomogeneous
boundary condition.   The key idea  is to consider $u - \Phi$ which satisfies homogeneous boundary conditions,  where $\Phi$ is the stochastic, incompressible background field (\ref{Background-vec}), carrying the inhomogeneities of the problem. One can then proceed \textit{formally} by taking the scalar product of the equation \eqref{SNSE} by $u - \Phi$ to obtain the following $\mathbb{P}$-a.s. energy inequality, 
\begin{equation}\label{E.I.1}
\int_0^T (d u, u) + \nu\, \int_0^T \|\nabla u\|^2 \, dt \leq \int_0^T (du , \Phi) + \int_0^T (u \cdot \nabla u , \Phi) \,dt + \nu \int_0^T (\nabla u , \nabla \Phi)\, dt. 
\end{equation}

 We present  the rest of the analysis  based on $ v= u - \Phi$ where $v$  is a fluctuating incompressible field which is unforced and hence of arbitrary amplitude. Making the substitution $u = v+  \Phi$ in  (\ref{SNSE}), we find  the stochastic process  $v $ satisfies, 
\begin{equation}\label{Eqn_v}
\begin{split}
d v + d \Phi = - (v \cdot \nabla v +  v \cdot \nabla \Phi + &\Phi \cdot\nabla v + \Phi \cdot\nabla \Phi - \nu \Delta v - \nu \Delta \Phi + \nabla p) \, dt,\\
& \grad \cdot v  =0,
\end{split}
\end{equation}
in the weak sense. The boundary conditions for $v$ are periodic in the $x_1$ and $x_2$ directions while in the $x_3$ direction, $$v(x_1, x_2, 0, t) =v(x_1,x_2, h, t)= 0\;.$$ 

From (\ref{E.I.1}), the energy-type inequality for $v$ is obtained as, 
\begin{equation}\label{Eq1}
\begin{split}
 \int_0^T  \underbrace{( v , dv)}_{\RN{1}} + \underbrace{(v , d\Phi )}_{\RN{2}}  \,  +  \, \nu \|\nabla v \|^2 \, dt \leq  \int_0^T  \big(&  \underbrace{( v \cdot \nabla v, v)}_{\RN{3}} +  \underbrace{( v \cdot \nabla \Phi, v)}_{\RN{4}} + \underbrace{(\Phi \cdot\nabla v, v)}_{\RN{5}} \\
& +  \underbrace{(\Phi \cdot\nabla \Phi, v)}_{\RN{6}}  + \underbrace{\nu (\nabla v, \nabla \Phi)}_{\RN{7}} \big)  \, dt.
\end{split}
\end{equation}
We shall estimate each numbered term in \eqref{Eq1}.

\subsection*{Term \RN{1}} \label{Term 1} 

Using (\ref{SNSE}) and (\ref{Background}) 
$$d v = d u - d \Phi = -  ( u \cdot \nabla u -  \nu \Delta u + \nabla p  ) \, dt - \left\{
\begin{array}{ll}
		\big(df(\X),\, 0,\, 0 \big)\tran  & \mbox{if } 0\leq x_3\leq \delta_t  \\
		0 & \mbox{otherwise}  
	\end{array}.
\right. $$

By  \eqref{Ito1}, the quadratic variation of  $f(\X)$ is $\int_0^t\sigma^2 (f'(\XX_s))^2 ds$.
Hence by It\^{o}'s product rule,
 \begin{equation} \label{vdv}
 v \cdot dv = \frac{1}{2} d (v\cdot v) -  \, \frac{\sigma^2}{2}  (f'(\X))^2 \,dt\;,   \quad \text{for} \quad  0\leq x_3\leq \delta_t.
 \end{equation}

Recall that the boundary layer $D_{\delta}= (0, L)^2 \times (0, \delta_t)$.
Using Proposition \ref{OUProperties} (ii) and \eqref{vdv} together with a direct calculation, we have
\begin{equation}
\begin{split}
 \int_{D} v \cdot dv \, dx  &= \frac{1}{2} d \|v\|^2 - \frac{\sigma^2}{2} \int_{D_{\delta}} (f'(\X))^2 \,dx\,  dt.
 \end{split}
\end{equation}

\subsection*{Term \RN{2}} \label{Term 2}
From \eqref{Ito1} it follows that 
\begin{equation}\label{Term2ab}
\begin{split}
\int_{D} v \, d\Phi \, dx  &= \int_{D_{\delta}} v_1\, d f(\X)  \, dx\\
&=\int_{D_{\delta}} v_1\, \mathcal{L}f(\X)\,dx   \, dt\,+\,
\sigma\int_{D_{\delta}} v_1 f'(\X) dx\,d W_t.
\end{split}
\end{equation}

\subsection*{Term \RN{3}} \label{Term 3} Using the incompressibility of $v$, along with integration by parts, we get 
$$(v \cdot \nabla v, \,v) =0.$$
\subsection*{Term \RN{4}} \label{Term 4}
Since $v_1$ vanishes on the bottom wall,  we can
write $v_1(x_1, x_2, x_3)$ as $\int_0^{x_3} \frac{\partial v_1}{\partial \zeta} (x_1, x_2, \zeta) \, d\zeta$.
Applying the Cauchy-Schwarz inequality (twice), we first estimate as
\begin{align*}
\left|\int_{0}^{L}\, \int_{0}^{L} v_1\, v_3 dx_1 dx_2 \right| & =
\left|\int_{0}^{L}\, \int_{0}^{L} \int_0^{x_3} \frac{\partial v_1}{\partial \xi}(x_1,x_2,\xi) \ d\xi \, 
 \int_0^{x_3} \frac{\partial v_3}{\partial \eta}(x_1,x_2,\eta)\ d \eta\, dx_1 dx_2\right| \\
&\leq x_3 \|\frac{\partial  v_1}{\partial  x_3}\|\,  \|\frac{\partial  v_3}{\partial  x_3}\|\;.
\end{align*}
Using this together with Young's inequality, we have
\begin{equation}
\begin{split}
|( v \cdot \nabla \Phi, v)| = \left|\int_{D_{\delta}} v_1 v_3 \frac{\partial \phi}{\partial x_3} dx\right| & \leq \left|\frac{\X}{\delta_t}\right|\, \left| \int_0^L \, \int_0^L\, \int_0^{\delta_t} v_1 v_3 \, dx_1 dx_2 dx_3\right|\\
& = \left|\frac{\X}{\delta_t}\right|\, \left| \int_0^{\delta_t} \, \left[\int_0^L\, \int_0^L v_1 v_3 \, dx_1 dx_2\right] dx_3\right|\\
& \leq \left|\frac{\X}{\delta_t}\right|\, \left| \int_0^{\delta_t} \, \left[ x_3 \|\frac{\partial  v_1}{\partial  x_3}\|\,  \|\frac{\partial  v_3}{\partial  x_3}\| \right] dx_3\right|\\
& =  \left|\frac{\X}{\delta_t}\right|\,   \frac{\delta_t ^2}{2} \, \|\frac{\partial  v_1}{\partial  x_3}\|\,  \|\frac{\partial  v_3}{\partial  x_3}\|\\
& \leq \frac{\delta_t}{2} |\X| \left[ \frac{1}{2} \|\frac{\partial  v_1}{\partial  x_3}\|^2 + \frac{1}{2} \|\frac{\partial  v_3}{\partial  x_3}\|^2\right]\\
&\leq  \frac{\delta_t}{2}\,  |\X| \, \|\nabla v\|^2.
\end{split}
\end{equation}

\subsection*{Term \RN{5}} \label{Term 5}  Using a  pointwise calculation we have 
$$\Phi \cdot \nabla v = \phi (x_3,\X) \frac{\partial v}{ \partial x_1}.$$ 
Therefore,  using  integration by parts and then the periodicity of $v$,  one can show that, 
\begin{equation}
\begin{split}
(\Phi \cdot \nabla v, v ) &= \frac{1}{2} \int_{D_{\delta}} \phi (x_3,\X) \frac{\partial}{\partial x_1} |v|^2 \, d x\\
& = \frac{1}{2} \int_0^{\delta_t} \phi(x_3,\X) \, \int_0^L  \left(\int_0^L \frac{\partial}{\partial x_1}   |v|^2  \, dx_1 \right)\,dx_2\,dx_3\\
& = 0.
\end{split}
\end{equation}

\subsection*{Term \RN{6}} \label{Term 6} A pointwise calculation leads to $\Phi \cdot \nabla  \Phi = 0$, hence,  
$$(\Phi \cdot \nabla \Phi, v ) =0.$$
\subsection*{Term \RN{7}} \label{Term 7} Direct calculation shows that  $\frac{\partial \phi(x_3,z)}{\partial x_3}=\frac{-z}{\delta(z)}$ for $0<x_3<\delta(z)$. Hence 
\begin{equation}\label{Gradphi}
    \|\frac{\partial \phi}{\partial x_3} \| = \frac{L}{\delta_t^{1/2}}\, |\X|.
\end{equation}

Therefore using the Cauchy-Schwarz inequality and Young's inequality, we find 
\begin{equation}\label{ineq_TermVII}
\begin{split}
&|\nu (\nabla v, \nabla \Phi)| \leq \nu \int_{D}  \left| \frac{\partial \phi}{\partial x_3}\right| \, 
\left| \frac{\partial v_1}{\partial x_3}\right| dx\\
& \leq \nu \| \frac{\partial \phi}{\partial x_3}\| \, \| \frac{\partial v_1}{\partial x_3}\| \\
& \leq \nu  \, \frac{L}{\delta_t^{1/2}}\, |\X|\, \|\nabla v\|\\
& \leq  \frac{\nu}{\delta_t} L^2 |\X|^2 + \frac{\nu}{4} \|\nabla v\|^2.
\end{split}
\end{equation}

Using  the estimates  for all the seven terms above in (\ref{Eq1}) yields,
\begin{equation}\label{Eq2}
\begin{split}
& \frac{1}{2} d \|v \|^2 + \frac{3\nu}{4} \|\nabla v \|^2 \, dt \,+\,\sigma\int_{D_{\delta_t}} v_1 f'(\X) dx\,d W_t\\
& \leq \, 
\frac{\sigma^2}{2} \int_{D_{\delta}} (f'(\X))^2 \,dx\,  dt + 
\left| 
\int_{D_{\delta}} v_1\, \mathcal{L}f(\X)\,dx   
\right| \,dt
+  \left[\frac{\delta_t}{2} |\X| \, \|\nabla v\|^2+ \nu L^2 \frac{|\X|^2}{\delta_t}   \right]\, dt,
\end{split}
\end{equation}
where we recall that $\delta_t=\delta(\X)$, the function $  f(z)=f_{x_3}(z)=\left(1- \dfrac{x_3}{\delta(z)}\right)z$ is defined in \eqref{Def:f} and therefore has derivatives given by \eqref{Df_0}. 

The second term on the right hand side of  (\ref{Eq2}) can be bounded from above by using the next lemma, which is proved in the Appendix. 
\begin{lemma}\label{L:BoundvG}
Let  $G=(G_t)_{t\in \mathbb{R}_+}$ be a stochastic process defined on the probability space in the martingale solution to \eqref{SNSE}-\eqref{BC}. Then $\mathbb{P}$-a.s., we have for all $t\in \mathbb{R}_+$,
\begin{equation*}
\left|\int_{D_\delta} v_1 \,G_t\, dx\right| 
  \leq  \,  \|\nabla v(t) \| \, \delta_t \, L\, \left(    \int_0^{\delta_t}  \left|G_t\right|^2   dx_3  \right)^{\frac{1}{2}}.
\end{equation*}
\end{lemma}

Applying Lemma \ref{L:BoundvG} with $G_t=\mathcal{L}f(\X)$ and then
 using Young's inequality, we have
\begin{equation}\label{Eq3}
\begin{aligned}
\left|\int_{D_\delta} v_1 \,\mathcal{L}f(\X)\, dx\right| & \leq    \|\nabla v \| \, \delta_t \, L\, \left(    \int_0^{\delta_t}  \left|\mathcal{L}f(\X)\right|^2   dx_3  \right)^{\frac{1}{2}}\\
& \leq \frac{\nu}{4} \, \|\nabla v \|^2 + \frac{1}{\nu} \delta_t^2\, L^2 \, \left(    \int_0^{\delta_t}  \left|\mathcal{L}f(\X)\right|^2   dx_3  \right)
\end{aligned}
\end{equation}

Hence  inserting estimate   (\ref{Eq3}) in (\ref{Eq2}), and  collecting terms that involve $\|\nabla v\|$,  we have the following stochastic equation. 
\begin{equation}\label{Eq2_v2}
\begin{split}
& \frac{1}{2} d \| v \|^2 + \left(\frac{1}{2} - \frac{\delta_t \, |\X|}{2 \nu}\right) \nu \| \nabla v \|^2 dt + \,\sigma\int_{D_{\delta_t}} v_1 f'(\X) dx\,d W_t  \\ 
\leq & 
 \left[\frac{\sigma^2}{2} \int_{D_{\delta}} (f'(\X))^2 \,dx + {\nu L^2 \frac{|\X|^2}{\delta_t} } +   \,\frac{1}{\nu} \delta_t^2\, L^2 \,     \int_0^{\delta_t}  \left|\mathcal{L}f(\X)\right|^2   dx_3  \right]\, dt.
\end{split}
\end{equation}
All stochastic differential inequalities  appearing in this paper should be interpreted in their corresponding integral forms. 

We note that the calculations up to and including \eqref{Eq2_v2} work for a general $C^2$ function $\delta=\delta(z)$.
For $\delta$ as in \eqref{delta} it is crucial to choose $A$ and $B$ such that $\left(\frac{1}{2} - \frac{\delta_t \, |\X|}{2 \nu}\right)$ in the second term of \eqref{Eq2_v2} to be strictly positive. Such conditions are summarized in the following lemma.
\begin{lemma}\label{L:IneqForDelta}
Let $\delta_t=\delta(\X)$, where $\X$ is a stochastic process in $\mathbb{R}$ and $\delta(z)= \frac{A}{z^2 + B}$.
Suppose $A$ and $B$ are positive numbers such that  $\frac{A}{B}<h$ and $A\leq \nu\sqrt{B}$. Then with probability one, for all $t\geq 0$ we have $\delta_t<h$ and
\begin{equation}\label{IneqForDelta}
\frac{1}{4}  \leq  \frac{1}{2} - \frac{\delta_t \, |\X|}{2 \nu} \leq \frac{1}{2}.  
\end{equation}
These hold if, for instance, $A=\nu U$ and $B=U^2$ and $\frac{U\, h}{\nu} >1$. 
\end{lemma}


\begin{proof}
Note that  $\delta_t \in (0,h)$ if $\frac{A}{B}<h$. Next, by the inequality $\frac{z}{z^2+B}\leq \frac{1}{2\sqrt{B}}$ for all $z\in\mathbb{R}$, we have
\begin{equation*}
\frac{1}{2}-\frac{A}{4\nu\,\sqrt{B}}  \leq  \frac{1}{2} - \frac{\delta_t \, |\X|}{2 \nu}.  
\end{equation*}
The term on the left is at least $1/4$ if  $A\leq \nu\sqrt{B}$.
\end{proof}


We summarize the above derivations in
 the following almost sure upper bound for the energy dissipation, which is the main result of this section.

\begin{lemma}\label{L:Eq6b}
Suppose  $A$ and $B$ are positive constants such that  $\frac{A}{B}<h$ and  $A\leq \nu\sqrt{B}$. 
Then with probability one, the following inequality holds for all $T>0$.
\begin{equation}\label{Eq6b}
 \int_0^T\,  \nu \| \nabla v \|^2 dt   + 4M_T \leq  2  \| v(0) \|^2 -2 \| v(T) \|^2+ Y_T,
\end{equation}
where 
\begin{equation}\label{M_t}
M_T:=  \sigma \int_0^T\int_{D_{\delta}} v_1\,\left(1-x_3\frac{3\X^2+B}{A}\right)\, dx\,d W_t.
\end{equation}
and 
\begin{equation}\label{Y_t}
Y_T:=  4\, L^2 \,T\, \left[  \frac{3}{2}\, \frac{A}{B} + \frac{6}{\nu}\, \left(\frac{A}{B}\right)^3 \, \frac{\sigma^2}{B}\right] \, \sigma^2 \,+\,
 4\, L^2 \int_0^T \left( \nu  \frac{|\X|^2}{\delta_t}  +   \,\frac{6}{\nu}\left(\frac{A}{B}\right)^{3} \,\theta^2\, |U - \X|^2 \, \right) dt.
\end{equation}
\end{lemma}

\begin{proof}
The stochastic integral term in \eqref{Eq2_v2} is
\[
\sigma\int_{D_{\delta}} v_1 f'(\X) dx\,d W_t  =
\sigma\int_{D_{\delta}} v_1\,\left(1-x_3\frac{3\X^2+B}{A}\right)\, dx\,d W_t.
\]

We now estimate terms on the right hand side of \eqref{Eq2_v2}.
For the first term, $\int_{D_{\delta}} (f'(\X))^2 \,dx = L^2\int_0^{\delta_t}(f'(\X))^2 \,dx_3$ and
\begin{align}
\int_{0}^{\delta_t} (f'(\X))^2 \,dx_3 = &\,
\int_{0}^{\delta_t} \left(1-x_3\frac{3\X^2+B}{A}\right)^2 \,dx_3 \notag\\
=&\, \delta_t- \delta_t^2\frac{3\X^2+B}{A} +   \frac{\delta_t^{3}}{3} \left(\frac{3\X^2+B}{A}\right)^2  \notag\\
\leq &\, \delta_t- \delta_t^2\frac{\X^2+B}{A} +   \frac{\delta_t^{3}}{3}\,\frac{9}{\delta_t^2}  \notag\\
\leq & \,3\,\frac{A}{B}, \label{int_f'}
\end{align}
where we used the fact that $\frac{3\X^2+B}{A}\leq \frac{3}{\delta_t}$ and $\delta(z)\leq \frac{A}{B}$ for all $z\in\mathbb{R}$.

Now we consider the term involving $\mathcal{L}f(\X)$. By the definition \eqref{Def:Generator} of $\mathcal{L}$ and the elementary inequality $(a+b)^2\leq 2(a^2+b^2)$,
\begin{align*}
\mathcal{L}f(\X)=&\, f'(\X)\theta (U - \X) + \frac{\sigma^2}{2}f''(\X) \\
\left|\mathcal{L}f(\X)\right|^2\leq&\,
2\,|f'(\X)|^2\,\theta^2\,(U-\X)^2\,+\,\frac{\sigma^4}{2}\left(f''(\X)\right)^2.
\end{align*}
So using  \eqref{int_f'} and the expression  $f''(\X)=-6x_3\frac{\X}{A}$,  we see that  in the last term on the right of \eqref{Eq2_v2},
\begin{align}
\int_0^{\delta_t} |\mathcal{L}f(\X)|^2\,dx_3  & \leq 2\,\theta^2\,(U-\X)^2\,\int_0^{\delta_t} (f'(\X))^2 \,dx_3 \,+\,\frac{\sigma^4}{2}\int_0^{\delta_t} \left(f''(\X)\right)^2\,dx_3 \notag\\
& \leq 6 \, \frac{A}{B} \,\theta^2\,(U-\X)^2 \,+\,6\sigma^4\,\frac{\X^2}{A^2}\,\delta_t^3 \notag\\
& \leq 6 \, \frac{A}{B} \,\theta^2\,(U-\X)^2 \,+\,6\sigma^4\,\frac{\delta_t^2}{A} \notag\\
& \leq 6 \, \frac{A}{B} \,\theta^2\,(U-\X)^2 \,+\,6\sigma^4\,\frac{A}{B^2}.
\end{align}
In the above, we used the fact that $|\X|^2 \leq \frac{A}{\delta_t}$ and $\delta_t \leq \frac{A}{B}$.

Hence after using $\delta_t \leq \frac{A}{B}$, the right hand side of \eqref{Eq2_v2} (ignoring $dt$) is bounded above by
\begin{equation}\label{LastTerm}
\frac{3\, \sigma^2}{2}\, L^2\,\frac{A}{B}+ \nu L^2 \frac{|\X|^2}{\delta_t}  +  \,  L^2 \, \left(  \frac{6}{\nu}\, \, \left(\frac{A}{B}\right)^3 \,\theta^2\,(U-\X)^2 \,+\, \frac{6}{\nu} \, \,\left(\frac{A}{B}\right)^3 \frac{\sigma^4}{B} \, \right)\;.
\end{equation}

Applying \eqref{IneqForDelta}  to the second term on the left of \eqref{Eq2_v2}, and \eqref{LastTerm} to the right of \eqref{Eq2_v2}, we obtain
\begin{equation}\label{Eq6}
\begin{split}
&\frac{1}{2}  \| v(T) \|^2  - \frac{1}{2}  \| v(0) \|^2+ \frac{1}{4}\, \int_0^T\,  \nu \| \nabla v \|^2 dt  
+ \sigma \int_0^T\int_{D_{\delta}} v_1\,\left(1-x_3\frac{3\X^2+B}{A}\right)\, dx\,d W_t \\
\leq &   L^2 \int_0^T \left( \frac{3}{2}  \frac{A}{B} \, \sigma^2+ \nu  \frac{|\X|^2}{\delta_t}  + \frac{6}{\nu}\, \, \left(\frac{A}{B}\right)^3 \,\theta^2\,(U-\X)^2 \,+\, \frac{6}{\nu} \, \,\left(\frac{A}{B}\right)^3 \frac{\sigma^4}{B} \,\right) dt  \\
= &   L^2 \,T\, \left[  \frac{3}{2}\, \frac{A}{B} + \frac{6}{\nu}\, \left(\frac{A}{B}\right)^3 \, \frac{\sigma^2}{B}\right] \, \sigma^2 \,+\,
 L^2 \int_0^T \left( \nu  \frac{|\X|^2}{\delta_t}  +   \,\frac{6}{\nu}\left(\frac{A}{B}\right)^{3} \,\theta^2\, |U - \X|^2 \, \right) dt.
\end{split}
\end{equation}

\end{proof}

Condition \eqref{Condition1u} ensures that the process $M$ defined in \eqref{M_t} is a  martingale.
\begin{lemma}\label{L:mtg}
The process $(M_t)_{t\geq 0}$ defined in \eqref{M_t} is a  martingale whose quadratic variation satisfies
\begin{align}\label{QuadM_t}
[ M ]_T
\leq & \, 3\sigma^2L^2\left(\frac{A}{B}\right)^{3} \int_0^T\,   \|\nabla v \|^2\,    d t\quad\text{ for} \quad T\geq 0.
\end{align}
\end{lemma}

\begin{proof}

Applying Lemma \ref{L:BoundvG} with $G_t=f'(\X)=1-x_3\frac{3\X^2+B}{A}$, and then  \eqref{int_f'}, we have
\begin{equation}\label{term2p_v2}
\begin{aligned}
\left|\int_{D_\delta} v_1 \,f'(\X)\, dx\right| 
  &\leq  \,  \|\nabla v \| \, \delta_t \, L\, \left(    \int_0^{\delta_t}  \left|f'(\X)\right|^2   dx_3  \right)^{\frac{1}{2}}\\
  &\leq \,  \|\nabla v \| \, \delta_t \, L\, \left(    \frac{3A}{B} \right)^{\frac{1}{2}}\\
  &\leq\, 3^{1/2}\,\|\nabla v \| \, L\, \left(\frac{A}{B}\right)^{3/2}.
\end{aligned}
\end{equation}
In the above, we used the fact that $\delta(z)\leq \frac{A}{B}$ for all $z\in\mathbb{R}$.

Hence the quadratic variation of $M_T$ is
\begin{align*}
[ M ]_T=& \sigma^2\int_0^T\left[\int_{D_{\delta}}  v_1\,\left(1-x_3\frac{3\X^2+B}{A}\right)\, dx \right]^2  dt \\
\leq &  \sigma^2\int_0^T\left[3^{1/2}\,\|\nabla v \|\,L\,\left(\frac{A}{B}\right)^{3/2}\right]^2  dt\\
\leq& 3\sigma^2\left(\frac{A}{B}\right)^{3}L^2 \int_0^T\,   \|\nabla v \|^2\,    d t.
\end{align*}
\end{proof}

\section{Estimation of the Mean Value}

To construct the estimate on  $\mathbb{E} [\langle\epsilon \rangle_T]$, we shall take the  expected value  of \eqref{Eq6b} with respect to $\mathbb{P}$, then  average it over $[0, T]$, and finally take the limit superior as $T \rightarrow \infty$.  Since $u = v + \Phi$, we obtain
\begin{equation}\label{Grad_uvPhi}
\int_0^T\,  \| \nabla u \|^2 dt = \int_0^T\,  \| \nabla v+ \nabla\Phi \|^2 dt 
\leq  2 \int_0^T\,  \| \nabla v \|^2  + \|\nabla \Phi\|^2 dt.
\end{equation}

The second term in the integrand is, from \eqref{Gradphi},
\begin{equation}\label{Gradphi2}
\|\nabla \Phi\|^2   =\|\frac{\partial \phi}{\partial x_3} \|^2 = \frac{L^2}{\delta_t}\, \X^2= L^2\,\frac{\X^4+B\X^2}{A}
\end{equation}
Hence
\begin{align}
\mathbb{E}\left[\int_0^T\,  \| \nabla \Phi \|^2 dt\right]  =\frac{T\,L^2}{A}\,\mathbb{E}\left[ \X^4+B\X^2 \right] 
\end{align}
which can be evaluated explicitly using \eqref{m1} and \eqref{m2} below. From Proposition \ref{OUProperties}, 
\begin{align}
\E\left[ |\X|^2 \right]=& U^2+\frac{\sigma^2}{2\theta},\qquad \qquad \qquad \E\left[ |U-\X|^2 \right]= \frac{\sigma^2}{2\theta}, \label{m1}\\
\E\left[ |\X|^4\right] =&
  U^4+6U^2\Big(\frac{\sigma^2}{2\theta}\Big)+3\Big(\frac{\sigma^2}{2\theta}\Big)^2,\qquad  \E\left[ |U-\X|^4\right] =3\Big(\frac{\sigma^2}{2\theta}\Big)^2.\label{m2}
\end{align}

We now estimate the first term on the right of
\eqref{Grad_uvPhi}.
From  Lemma \ref{L:mtg}, $M_{T}$ is a martingale and hence
\begin{equation}\label{Eq6-6}
\E[M_{T}]  =0 \quad\text{for all }T\in [0,\infty).
\end{equation}
Therefore, taking the expectation $\mathbb{E}$ of both sides of
\eqref{Eq6b} gives
\begin{equation}\label{Eq6b_taun3}
 \mathbb{E}\int_0^{T}\,  \nu \| \nabla v \|^2 dt   \leq  \mathbb{E} \left[ 2  \| v(0) \|^2 + Y_{T} \right].
\end{equation}

We shall estimate the expectation of the integral term in $Y_T$ defined in \eqref{Y_t}. 
To this end we need 
some standard properties for the stationary OU process  and Gaussian random variables as stated in Proposition \ref{OUProperties}.
Recall that,  $\X$ has  normal distribution with mean $U$ and variance $\frac{\sigma^2}{2\theta}$ for all $t\in\mathbb{R}_+$ under $\mathbb{P}$. Hence $U - \X$ is a centered normal variable with variance $\frac{\sigma^2}{2\theta}$.

Hence we
can compute the expectation of the integral of  (\ref{Y_t}) as follows.
\begin{align}
&\E\int_0^T \left(\nu  \frac{|\X|^2}{\delta_t}  +   \,\frac{6}{\nu}\left(\frac{A}{B}\right)^{3} \,\theta^2\, |U - \X|^2 \right) \, dt \notag\\
=&\E\int_0^T  \left( \nu  \frac{\X^4+B\X^2}{A}  +   \,\frac{6}{\nu}\left(\frac{A}{B}\right)^{3} \,\theta^2\, |U - \X|^2 \, dt \right) \notag\\
=& T\,\left\{\frac{\nu}{A}\left(U^4+6U^2\Big(\frac{\sigma^2}{2\theta}\Big)+3\Big(\frac{\sigma^2}{2\theta}\Big)^2\,+BU^2+B\frac{\sigma^2}{2\theta}\right)\,+\,\frac{6}{\nu}\left(\frac{A}{B}\right)^{3} \,\theta^2\,\frac{\sigma^2}{2\theta}\right\}. \label{MeanIntegral}
\end{align}

Now we continue from \eqref{Eq6b_taun3}.
Divide both sides by $T$ and $|D| = L^2 h$, and use \eqref{MeanIntegral} to obtain 
\begin{align}
& \limsup_{T\to\infty}\frac{1}{TL^2h} \,\mathbb{E}\int_0^{T}\,  \nu \| \nabla v \|^2 dt  \notag \\
\leq &\, \limsup_{T\to\infty}\frac{1}{TL^2h} \,\mathbb{E}[Y_{T}] \notag\\
= &\, \frac{4}{h}\, \left[ \frac{3}{2}  \frac{A}{B}  +   \,\frac{6}{\nu}
\frac{\sigma^2}{B}\left(\frac{A}{B}\right)^{3} \right] \, \sigma^2\notag\\
&\,+\,
\frac{4}{h}\left\{\frac{\nu}{A}\left[U^4+6U^2\Big(\frac{\sigma^2}{2\theta}\Big)+3\Big(\frac{\sigma^2}{2\theta}\Big)^2+BU^2+B\frac{\sigma^2}{2\theta}\right]\,+\,\frac{6}{\nu}\left(\frac{A}{B}\right)^{3} \frac{\sigma^2\theta}{2}\right\}.
\label{MeanUpper}
\end{align}

Finally, by \eqref{Grad_uvPhi} and (\ref{MeanUpper}), one obtains the estimate 
 \begin{align*}
\varepsilon \leq &\,  \limsup_{T\to\infty}\frac{2}{TL^2h} \,\mathbb{E}\int_0^{T}\,  \nu \| \nabla v \|^2 dt + \limsup_{T\to\infty}\frac{2}{TL^2h} \,\mathbb{E}\int_0^{T}\,  \nu \| \nabla \Phi \|^2 dt \\
\leq  &\, \frac{8}{h}\, \left[ \frac{3}{2}  \frac{A}{B}  +   \,\frac{6}{\nu}
\frac{\sigma^2}{B}\left(\frac{A}{B}\right)^{3} \right] \, \sigma^2 \\
&\quad +\,
\frac{8}{h}\left\{\frac{2\nu}{A}\left[U^4+6U^2\Big(\frac{\sigma^2}{2\theta}\Big)+3\Big(\frac{\sigma^2}{2\theta}\Big)^2+BU^2+B\frac{\sigma^2}{2\theta}\right]\,+\,\frac{6}{\nu}\left(\frac{A}{B}\right)^{3} \frac{\sigma^2\theta}{2}\right\}.
\end{align*}

Taking $A = \nu U$, $B = U^2$ (which seem to be nearly optimal), in terms of the Reynolds number  $\Rey=\frac{Uh}{\nu}$ , the above estimate can be written as,
\begin{align}\label{inRe}
\varepsilon \leq 32 \,  \frac{U^3}{h} + 2\left( 6  \, \frac{1}{\Rey} + 28 \, \frac{U}{ h \, \theta} + 12\, \frac{1}{\Rey^2} \,  \frac{h \, \theta}{U} +  24 \frac{1}{\Rey^2} \, \frac{h \, \sigma^2}{U^3} + 6\, \frac{\sigma^2}{h \, U \, \theta^2}\right)\, \sigma^2\;.
\end{align}

\begin{re}[Large noise regime]\label{altRe}
When the noise is large compared to the mean of the OU process, $U$, one might interpret our estimate in terms of the alternative characteristic velocity  $\widetilde U=\sigma/\sqrt{\theta}$ as
\begin{align*}
    \varepsilon &\lesssim \frac{1}{h}\left\{ \left[\frac{h}{\tRe}\frac{\tU}{U} + \frac{h\theta}{\tRe^2}\frac{\tU^4}{U^5}\right]\theta\tU^2 + 
    U^3+U\tU^2+\frac{\tU^4}{U} + 
    +\frac{h^2\theta^2}{\tRe^2} \frac{\tU^4}{U^3}\right\} \\
    & \sim \frac{1}{h}\left\{ U^3 + U\tU^2+\frac{\tU^4}{U}\right\}
 \quad \text{for large} \quad \tRe=\tU h/\nu\;.
\end{align*}
\end{re}

\begin{re}[Over-dissipation]\label{Remark5.1}
 If, in our analysis, we were to instead take $\X$ to be Brownian motion, i.e., $\X = W_t$,  this would result in a potential over-dissipation of the model, since,
$$  \frac{1}{T}\, \E\left[\int_0^T\,  |\X|^2 \, dt \right] = \frac{1}{T}\, \int_0^T\, \E\left[ W_t^2\right] \, dt   = \frac{1}{2} T \rightarrow \infty,  \hspace{1cm} \text{as} \hspace{1cm}  T\rightarrow \infty. $$
\end{re}

\begin{re}\label{Remark2}
If $\theta \rightarrow 0$, the estimate in \eqref{inRe} tends to infinity. Roughly speaking, this potential over-dissipation of the model is consistent with Remark \ref{Remark5.1}. This because as $\theta \rightarrow 0$,  the OU process \eqref{OU} tends to $\sigma W$ which is a Wiener process with a constant time-change.
\end{re}

 \section{Estimation of higher moments}

To estimate higher moments of 
\begin{equation*}\label{Def:timeaverageT}
\langle \epsilon\rangle_T=\frac{1}{|D|}\,\frac{1}{T}\int_{0}^{T}   \nu\|\nabla u(t ,\cdot, \omega)\|_{L^2(D)}^2  \, dt \;,
\end{equation*}
we shall need higher moments of the stationary OU process $\X$. By
Proposition \ref{OUProperties},
\begin{align}
 \E\left[ |\X|^6\right] =&U^6+15U^4\frac{\sigma^2}{2\theta}+45U^2\left(\frac{\sigma^2}{2\theta}\right)^2+15\left(\frac{\sigma^2}{2\theta}\right)^3, \label{m3}\\
\E\left[ |\X|^8\right] =&U^8+28U^6\frac{\sigma^2}{2\theta}+210U^4\left(\frac{\sigma^2}{2\theta}\right)^2 
  +420U^2\left(\frac{\sigma^2}{2\theta}\right)^3
  +105\left(\frac{\sigma^2}{2\theta}\right)^4.\label{m4}
\end{align}
More generally, for all integer $k\geq 1$, $\E\left[ |\X|^{2k}\right]=U^{2k}+P_k(U^2,\,\sigma^2/(2\theta))$for some polynomial $P_k$.


By \eqref{Grad_uvPhi}, for all $p\in[1,\infty)$,
\begin{align}
\mathbb{E}\left[\left|\int_0^T\,  \| \nabla u \|^2 dt\right|^p\right]  
\leq &\, 2^p \,\mathbb{E}\left[\left|\int_0^T\,   \| \nabla v \|^2 +\|\nabla \Phi\|^2 dt\right|^p\right] \notag\\
\leq &\, 4^p\,\left(\mathbb{E}\left[\left|\int_0^T\,  \| \nabla v \|^2 dt\right|^p\right]+\mathbb{E}\left[\left|\int_0^T\,  \| \nabla \Phi \|^2 dt\right|^p\right]\right). \label{Var_u}
\end{align}

To bound the second term on the right of \eqref{Var_u},  from H\"older's inequality we have 
\begin{align}
\mathbb{E}\left[\left|\int_0^T\,  \| \nabla \Phi \|^2 dt\right|^p\right]  \leq &\,
T^{p-1}\,\mathbb{E}\left[\int_0^T\,  \| \nabla \Phi \|^{2p} dt\right]. \label{p_Holder}
\end{align}

We shall focus on the case $p=2$, even though our estimates below can be extended to any $p\in [1,\infty)$. From \eqref{p_Holder} and \eqref{Gradphi2}
\begin{align} \label{GradPhi_2nd}
&\mathbb{E}\left[\left|\int_0^T\,  \| \nabla \Phi \|^2 dt\right|^2\right]  \leq 
T\,\mathbb{E}\left[\int_0^T\,  \| \nabla \Phi \|^{4} dt\right] 
\leq \frac{2\,T^2\,L^4}{A^2}\,\mathbb{E}\left[  \X^8+B^2\X^4\right]
\end{align}
which can be computed explicitly using the moment formulas \eqref{m2} and \eqref{m4} for the OU process.

For the first term on the right of \eqref{Var_u}, we write
\begin{equation}\label{Def:E_T}
\mathcal{E}_T:=  \int_0^T\,  \nu \| \nabla v \|^2 dt.
\end{equation}
Lemma \ref{L:Eq6b} asserts that
\begin{equation}\label{Eq6b1}
\mathcal{E}_T \leq  2  \| v(0) \|^2 + Y_T +  |M_T|.
\end{equation}
Hence
\begin{equation}\label{Eq6b2}
\mathbb{E}\left[\left|\mathcal{E}_T\right|^2\right]  \leq  3\,\mathbb{E}\left[4  \| v(0) \|^2 + |Y_T|^2 +  |M_T|^2\right].
\end{equation}

\subsection{Bounding $\mathbb{E}[M_T^2]$}
By Lemma \ref{L:mtg}, Jensen's inequality and then Young's inequality, we obtain
\begin{align}\label{QuadM_t2}
    3\mathbb{E}[M_T^2]=3\mathbb{E}[[M]_T] \leq &
   \alpha\,  \mathbb{E}[\mathcal{E}_T] \notag\\
    \leq & \alpha\sqrt{\mathbb{E}[|\mathcal{E}_T|^2]} \leq \frac{\alpha^2}{2}+\frac{\mathbb{E}[|\mathcal{E}_T|^2]}{2}
\end{align}
where 
$\alpha= \frac{9\sigma^2L^2}{\nu}\left(\frac{A}{B}\right)^{3}$ has the same dimension as that of $\mathcal{E}_T$ when we choose $A=\nu U$ and $B=U^2$.

\subsection{Bounding $\mathbb{E}[|Y_T|^2]$}
We apply the elementary inequality $(a+b)^2\leq 2(a^2+b^2)$ and the Cauchy-Schwarz inequality to \eqref{Y_t} and to obtain 
\begin{align}\label{Y_t2}
|Y_T|^2 \leq & \, 
32\,L^4\,T^2\,\left[  \frac{3}{2}\, \frac{A}{B} + \frac{6}{\nu}\, \left(\frac{A}{B}\right)^3 \, \frac{\sigma^2}{B}\right]^2 \, \sigma^4 \notag\\
&\,+\, 32\,L^4\,T\, \int_0^T\left( \nu  \frac{|\X|^2}{\delta_t}  +   \,\frac{6}{\nu}\left(\frac{A}{B}\right)^{3} \,\theta^2\, |U - \X|^2  \right)^2\, dt.
\end{align}
Applying  $(a+b)^2\leq 2(a^2+b^2)$ again, the integrand in the second term is bounded above by
\begin{align*}
\left( \nu  \frac{|\X|^2}{\delta_t}  +   \,\frac{6}{\nu}\left(\frac{A}{B}\right)^{3} \,\theta^2\, |U - \X|^2  \right)^2 
\leq&  2 \left(
 \nu^2  \frac{(\X^4+B\X^2)^2}{A^2}  +   \,\frac{36}{\nu^2}\left(\frac{A}{B}\right)^{6} \,\theta^4\, |U - \X|^4 \right)\\
\leq & 
 \frac{4\nu^2}{A^2}  (\X^8+B^2\X^4)  +   \,\frac{72}{\nu^2}\left(\frac{A}{B}\right)^{6} \,\theta^4\, |U - \X|^4 
\end{align*}

Hence
\begin{align}
\mathbb{E}[|Y_T|^2] \leq & \,
32\,L^4\,T^2\,\left[  \frac{3}{2}\, \frac{A}{B} + \frac{6}{\nu}\, \left(\frac{A}{B}\right)^3 \, \frac{\sigma^2}{B}\right]^2 \, \sigma^4 \notag\\
&\,+\, 32\,L^4\,T^2\,\left(
 \frac{4\nu^2}{A^2}  \mathbb{E}[\X^8+B^2\X^4]  +   \,\frac{72}{\nu^2}\left(\frac{A}{B}\right)^{6} \,\theta^4\,\mathbb{E}[ |U - \X|^4 ]
\right). \label{Y_t3}
\end{align}

\subsection{Summarizing}
Putting \eqref{QuadM_t2} into \eqref{Eq6b2}, we obtain
\begin{align*}
\mathbb{E}[|\mathcal{E}_T|^2]  \leq &  12\,\mathbb{E}\left[\| v(0) \|^2\right] + \left(\frac{\alpha^2}{2}+\frac{\mathbb{E}[|\mathcal{E}_T|^2]}{2}\right)\,+\,3\mathbb{E}[|Y_T|^2].
\end{align*}
Rearranging terms gives
\begin{align}\label{Eq6b4}
\mathbb{E}[|\mathcal{E}_T|^2]  \leq &  24\,\mathbb{E}\left[\| v(0) \|^2\right] + \alpha^2\,+\,6\mathbb{E}[|Y_T|^2].
\end{align}
Combining \eqref{Eq6b4} with \eqref{Var_u} (with $p=2$) gives
\begin{align*}
\mathbb{E}\left[\left|\int_0^T\,  \nu \| \nabla u \|^2 dt\right|^2\right]  
\leq &\, 16\,\left(\mathbb{E}\left[|\mathcal{E}_T|^2\right]+\mathbb{E}\left[\left|\int_0^T\,  \nu \| \nabla \Phi \|^2 dt\right|^2\right]\right)\\
\leq &\,384 \,\mathbb{E}[\|v(0)\|^2] +16\alpha^2 +96\,\mathbb{E}[|Y_T|^2] +16 \,\mathbb{E}\left[\left|\int_0^T\,  \nu \| \nabla \Phi \|^2 dt\right|^2\right].
\end{align*}
Hence using
\eqref{GradPhi_2nd} and \eqref{Y_t3}, and recalling $|D|=L^2 h$, we have
\begin{equation}\label{varianceboundexplicit}
\begin{aligned}
\limsup\limits_{T\rightarrow\infty}\mathbb{E}[\langle \epsilon\rangle_T^2] 
\,\leq &\, 16 \limsup_{T\to\infty}\frac{\mathbb{E}[|\mathcal{E}_T|^2]}{|D|^2T^2}  + 16 \limsup_{T\to\infty}\frac{1}{|D|^2T^2} \mathbb{E}\left[\left|\int_0^T\,  \nu \| \nabla \Phi \|^2 dt\right|^2\right]\\
\leq &\,\frac{96}{|D|^2}\bigg\{
32\,L^4\,\left[  \frac{3}{2}\, \frac{A}{B} + \frac{6}{\nu}\, \left(\frac{A}{B}\right)^3 \, \frac{\sigma^2}{B}\right]^2 \, \sigma^4 \notag\\
&\qquad \,+\, 32\,L^4\,\left(
 \frac{4\nu^2}{A^2}  \mathbb{E}[\X^8+B^2\X^4]  +   \,\frac{72}{\nu^2}\left(\frac{A}{B}\right)^{6} \,\theta^4\,\mathbb{E}[ |U - \X|^4 ]
\right) \bigg\}\\
&\,+\frac{16 \nu^2}{|D|^2}\,\frac{2\,L^4}{A^2}\,\mathbb{E}\left[  \X^8+B^2\X^4\right]\\
\leq &\,\frac{3072}{h^2}\bigg\{
\left[  \frac{3}{2}\, \frac{A}{B} + \frac{6}{\nu}\, \left(\frac{A}{B}\right)^3 \, \frac{\sigma^2}{B}\right]^2 \, \sigma^4 \notag\\
&\qquad \,+\, \left(
 \frac{4\nu^2}{A^2}  \mathbb{E}[\X^8+B^2\X^4]  +   \,\frac{72}{\nu^2}\left(\frac{A}{B}\right)^{6} \,\theta^4\,\mathbb{E}[ |U - \X|^4 ]
\right) \bigg\}\\
&\,+\frac{32 \nu^2}{h^2 A^2}\,\mathbb{E}\left[  \X^8+B^2\X^4\right]\\
\leq &\,\frac{3072}{h^2}\bigg\{
\left[  \frac{3}{2}\, \frac{A}{B} + \frac{6}{\nu}\, \left(\frac{A}{B}\right)^3 \, \frac{\sigma^2}{B}\right]^2 \, \sigma^4 \,+\, 
  \frac{72}{\nu^2}\left(\frac{A}{B}\right)^{6} \,\theta^4\,\mathbb{E}[ |U - \X|^4 ]
\bigg\}\\
&\,+\frac{12320 \nu^2}{h^2 A^2}\,\mathbb{E}\left[  \X^8+B^2\X^4\right].
\end{aligned}
\end{equation}
Now applying the moment formulas \eqref{m2}, \eqref{m4} and setting $A=\nu U$ and $B=U^2$, the above upper bound is
\begin{equation}
\begin{aligned}
&\,\limsup\limits_{T\rightarrow\infty}\mathbb{E}[\langle \epsilon\rangle_T^2] \\
\leq \,&\frac{3072}{h^2}\bigg\{
\left[  \frac{3}{2}\, \frac{A}{B} + \frac{6}{\nu}\, \left(\frac{A}{B}\right)^3 \, \frac{\sigma^2}{B}\right]^2 \, \sigma^4 \,+\, 
  \frac{216}{\nu^2}\left(\frac{A}{B}\right)^{6} \,\theta^4\,\Big(\frac{\sigma^2}{2\theta}\Big)^2
\bigg\}\\
&\,+\frac{12320 \nu^2}{h^2 A^2}\,
\bigg\{
\left[U^8+28U^6\frac{\sigma^2}{2\theta}+210U^4\left(\frac{\sigma^2}{2\theta}\right)^2 
  +420U^2\left(\frac{\sigma^2}{2\theta}\right)^3
  +105\left(\frac{\sigma^2}{2\theta}\right)^4\right]\\
 & \qquad\qquad\qquad + B^2\left[ U^4+6U^2\Big(\frac{\sigma^2}{2\theta}\Big)+3\Big(\frac{\sigma^2}{2\theta}\Big)^2\right]
\bigg\}\\
\leq \,&\frac{3072}{h^2}\bigg\{
\left[  \frac{3}{2}\, \frac{\nu}{U} + \frac{6\sigma^2 \nu^2}{U^5}\,\right]^2 \, \sigma^4 \,+\, 
  \frac{216 \nu^4}{U^6} \,\theta^4\,\Big(\frac{\sigma^2}{2\theta}\Big)^2
\bigg\}\\
&\,+\frac{12320}{h^2 U^2}\,
\bigg\{
\left[2U^8+34U^6\frac{\sigma^2}{2\theta}+213U^4\left(\frac{\sigma^2}{2\theta}\right)^2 
  +420U^2\left(\frac{\sigma^2}{2\theta}\right)^3
  +105\left(\frac{\sigma^2}{2\theta}\right)^4\right]\\
= \,& \frac{24640 U^6}{h^2} + \sigma^2 \,P_{U,\nu,\theta}(\sigma)
\end{aligned}
\end{equation}
where $P_{U,\nu,\theta}(\sigma)$ is an explicit polynomial in $\sigma$ whose coefficients are explicit functions of  $U,\nu,\theta$.

\begin{re}[Higher moments]\label{Rk:higher}
One can readily obtain estimates for higher moments  by following our method. Note that 
\eqref{Var_u} and \eqref{p_Holder} still hold, and 
for all integer $k\geq 1$, $\E\left[ |\X|^{2k}\right]=U^{2k}+P_k(U^2,\,\sigma^2/(2\theta))$ for some polynomial $P_k$. For the martingale term
\eqref{QuadM_t2},
one can apply the  Burkholder-Davis-Gundy Inequality. We expect that for all integer $k\geq 1$, the $2k$-th moment  of $\langle \epsilon\rangle_T$ satisfies
\begin{equation}
\begin{aligned}\label{Higherbound}
\limsup\limits_{T\rightarrow\infty}\mathbb{E}[\langle \epsilon\rangle_T^{2k}]   
\lesssim
\frac{U^{6k}}{h^{2k}} +  \sigma^2 \,P(\sigma)
\end{aligned}
\end{equation}
where $P(\sigma)=P_{U,\nu,\theta}(\sigma)$ is an explicit polynomial in $\sigma$ whose coefficients are explicit functions of $U,\nu,$ and $\theta$.
\end{re}

\section{Conclusion and Commentary}

In this paper we have derived uniform (in $T$) bounds for both the mean and the second moment of the energy dissipation rate for solutions of the incompressible  Navier--Stokes equations with a boundary wall moving as a stationary Ornstein--Uhlenbeck process.  As the variance of the OU process tends to 0, we recover an upper bound for the deterministic case as in \cite{DC92}.  A similar argument can be used to find higher moment bounds. A novelty of our method is the construction of a carefully chosen stochastic background flow $\Phi$ that depends on the stochastic forcing, as indicated in \eqref{delta}. Our technique can be readily generalized to obtain bounds for higher moments and to the case where the OU process is replaced by a gradient system of the form
\begin{equation}\label{gradientSDE}
dX_t=-\nabla h(X_t)\,dt +\sigma\,dW_t,
\end{equation}
where  $h$ is a function and $\sigma\in \mathbb{R}$. The OU process \eqref{OU} is the case where $h(x)=-\theta(x-U)^2/2$. It is well-known that if $$Z^{(\sigma)}:=\int_{\mathbb{R}} \exp{ \left( \frac{-2}{\sigma^2} h (x) \right)} \,dx <\infty\;,$$ then the
1-dimensional gradient system \eqref{gradientSDE}
has a unique invariant distribution given by the Gibbs measure 
\begin{equation}\label{Gibbs_1Dim}
    \frac{1}{Z^{(\sigma)}}\exp{ \left( \frac{-2}{\sigma^2} h (x) \right)} .
\end{equation}
The analysis herein would allow for over-dissipation of the model if the noise at the boundary were taken to be the Wiener process, as noted in Remarks \ref{Remark5.1} and \ref{Remark2}. 

Finally,  it was crucial to take the limit superior in time {\it after} the expectation.  Our estimate does not provide a bound when the operations are taken in the reverse order. It remains to find a bound in the latter case, or quantify the difference in the two expressions describing the rate of dissipation.
\section{Acknowledgments}
The work of M. Jolly was supported in part by  NSF grant DMS-1818754.
W.T. Fan is partially supported by NSF grant DMS-1804492 and ONR grant TCRI N00014-19-S-B001.
\section{Appendix}
Proof of Lemma \ref{L:BoundvG}: we first write $v_1(x_1, x_2, x_3)$ as $\int_0^{x_3} \frac{\partial v_1}{\partial \zeta} (x_1, x_2, \zeta) \, d\zeta$, and  apply the Cauchy-Schwarz inequality  to obtain 
\begin{equation}\label{term2p}
\begin{aligned}
\left|\int_{D_\delta} v_1 \,G_t\, dx\right|
&= \left| \int_0^L \int_0^L  \int_0^{\delta_t} G_t \,  v_1\,  dx_3 dx_2 dx_1 \right|\\
& = \left| \int_0^L \int_0^L  \int_0^{\delta_t} G_t \, \left( \int_0^{x_3} \frac{\partial v_1}{\partial \eta} (x_1, x_2, \eta) d \eta \right)\,  dx_3 dx_2 dx_1 \right|\\
& = \left| \int_0^L \int_0^L  \int_0^{\delta_t}  \int_0^{x_3} G_t \frac{\partial v_1}{\partial \eta} (x_1, x_2, \eta) \,  d \eta   dx_3  dx_2  dx_1 \right|\\
& \leq \left(  \int_0^L \int_0^L  \int_0^{\delta_t}  \int_0^{x_3} \left|G_t\right|^2  \,  d \eta   dx_3  dx_2  dx_1\right)^{\frac{1}{2}} \, \left(  \int_0^L \int_0^L  \int_0^{\delta_t}  \int_0^{x_3}  \left|\frac{\partial v_1}{\partial \eta}\right|^2 \,  d \eta   dx_3  dx_2  dx_1\right)^{\frac{1}{2}}
\end{aligned}
\end{equation}

Now we estimate the terms on the right hand side of (\ref{term2p}) as, 
\begin{equation*}
\begin{aligned}
\left(  \int_0^L \int_0^L  \int_0^{\delta_t}  \left( \int_0^{x_3} \left|G_t\right|^2  \,  d \eta \right)  dx_3  dx_2  dx_1\right)^{\frac{1}{2}}  & \leq \left(  \int_0^L \int_0^L  \int_0^{\delta_t} \left(  \left|G_t\right|^2 \int_0^{x_3} 1   \,  d \eta  \right) dx_3  dx_2  dx_1\right)^{\frac{1}{2}} \\
& =  L\,\left(  \int_0^{\delta_t}   \left|G_t\right|^2 \,x_3 \, dx_3  \right)^{\frac{1}{2}} \\
& \leq \delta_t^{\frac{1}{2}} \, L \,  \left(    \int_0^{\delta_t}  \left|G_t\right|^2   dx_3  \right)^{\frac{1}{2}}
\end{aligned}
\end{equation*}
and, 
\begin{equation*}
\begin{aligned}
\left(  \int_0^L \int_0^L  \int_0^{\delta_t}  \int_0^{x_3}  \left|\frac{\partial v_1}{\partial \eta}\right|^2 \,  d \eta   dx_3  dx_2  dx_1\right)^{\frac{1}{2}} & =   \left(  \int_0^{\delta_t} \int_0^L  \int_0^L  \int_0^{x_3}  \left|\frac{\partial v_1}{\partial \eta}\right|^2 \,  d \eta   dx_2  dx_1  dx_3\right)^{\frac{1}{2}}\\
& \leq \left(  \int_0^{\delta_t} \int_0^L  \int_0^L  \int_0^{\delta_t}  \left|\frac{\partial v_1}{\partial \eta}\right|^2 \,  d \eta   dx_2  dx_1  dx_3\right)^{\frac{1}{2}}\\
& \leq \delta_t^{\frac{1}{2}}\, \|\nabla v \|. 
\end{aligned}
\end{equation*}
Plugging the above two estimates in (\ref{term2p}) yields the desired inequality.


\begin{thebibliography}{9}

\bib{B01}{book}{
   author={V. Barbu},
 
   title={Stabilization of Navier--Stokes Flows},

   publisher={Springer, London}
   date={2001},
 
 
}


\bib{BCPW19}{article}{
   author={Bedrossian, J.},
      author={Coti Zelati, M.},
       author={ Punshon-Smith, S.},
        author={Weber, F.},

 TITLE = {A sufficient condition for the {K}olmogorov 4/5 law for
              stationary martingale solutions to the 3{D} {N}avier-{S}tokes
              equations},
   JOURNAL = {Comm. Math. Phys.},
  FJOURNAL = {Communications in Mathematical Physics},
    VOLUME = {367},
      YEAR = {2019},
    NUMBER = {3},
     PAGES = {1045--1075},
  
}




 \bib{BT73}{article}{
   author={A. Bensoussan},
      author={R. Temam},

   title={ Equatios stochastique du type Navier-Stokes},
   journal={J. Funct. Anal.},
   volume={13},
   date={1973},

   pages={195–222},
  
}



 \bib{BJMT14}{article}{
   author={A. Biswas},
      author={M. S. Jolly},
       author={V. R. Martinez},
         author={   E. S. Titi},
   title={Dissipation length scale estimates for turbulent flows: A Wiener
algebra approach},
   journal={Journal of Nonlinear Science},
   volume={24},
   date={2014},

   pages={441-471},
  
}




 \bib{B00}{article}{
   author={Breckner, H.},
     
   title={Galerkin approximation and the strong solution of the
              Navier-Stokes equation},
   journal={Journal of Applied Mathematics and Stochastic Analysis},
   volume={13},
   date={2000},
    NUMBER = {3},
     PAGES = {239--259},

}

\bib{BH10}{book}{
   author={Brezis, H.},
   title={Functional analysis, Sobolev spaces and partial differential equations},
 PUBLISHER = {Springer Science and Business Media},
      YEAR = {2010},
  
}

 \bib{BP00}{book}{
   author={Brze\'{z}niak, Z},
      author={Peszat, S},
     
   title={Infinite dimensional stochastic analysis},
    SERIES = {Verh. Afd. Natuurkd. 1. Reeks. K. Ned. Akad. Wet.},
    VOLUME = {52},
     PAGES = {85--98},
 PUBLISHER = {R. Neth. Acad. Arts Sci., Amsterdam},
      YEAR = {2000},


  
}









\bib{B70}{article}{
   author={F. H. Busse},
   title={Bounds for turbulent shear flow},
  
   journal={Journal of Fluid Mechanics},
   volume={41},
   date={1970},
 
   pages={4219--240},
   
}
\bib{CamassaKilicMcLaughlin19}{article}{
   author={R. Camassa},
   author={Z. Kilic},
   author={R. M. McLaughlin},
   title={On the symmetry properties of a random passive scalar with and
   without boundaries, and their connection between hot and cold states},
   journal={Phys. D},
   volume={400},
   date={2019},
   pages={132124, 32},
   issn={0167-2789},
   review={\MR{4008029}},
   doi={10.1016/j.physd.2019.05.004},
}

\bib{CLM01}{article}{
   author={T. Caraballo},
   author={K. Liu},
      author={X.R. Mao},
   title={On stabilization of partial differential equations by noise},
   journal={Nagoya Math. J.},
   volume={161(2)},
   date={2001},
   pages={155–170},
  
}




 \bib{C73}{article}{
   author={Chorin, A. J.},

   title={Numerical study of slightly visous flow},
   journal={J. Fluid Mech.},
   volume={57},
   date={1973},

   pages={785–796},
  
}

\bib{CP20}{article}{
   author={Y. T. Chow},
 author={A. Pakzad},
  title={On the zeroth law of turbulence for the stochastically forced Navier-Stokes equations},
   journal={arXiv:2004.08655},
 date={2020},
 
   }


\bib{DLPRSZ18}{article}{
   author={V. DeCaria},
   author={W. Layton},
    author={A. Pakzad},
     author={Y. Rong},
      author={N. Sahin},
       author={H. Zhao},
   title={On the determination of the grad-div criterion},
   journal={Journal of Mathematical Analysis and Applications},
   volume={467},
   issue={2},
   date={2018},
   pages={1032--1037},
   }
   
  \bib{D09}{article}{
   author={Doering, C .R.},
 
   title={The 3D Navier-Stokes Problem},
  
   journal={Annual Review of Fluid Mechanics},
   volume={41},
   date={2009},
   pages={109-128},
   
}
  
  
  

\bib{DC92}{article}{
   author={Doering, C .R.},
   author={Constantin, P.},
   title={Energy dissipation in shear driven turbulence},
  
   journal={Physical review letters},
   volume={69},
   date={1992},
   pages={1648},
   
}


 
   \bib{DC96}{article}{
   author={C. R. Doering},
   author={P. Constantin},
   title={Variational bounds on energy dissipation in incompressible flows. III. Convection},
  
   journal={Phys. Rev. E},
   volume={53},
   date={1996},
   pages={5957--5981},
   
}

\bib{DF02}{article}{
   author={C. R. Doering},
   author={C. Foias},
   title={Energy dissipation in body-forced turbulence},
   journal={J. Fluid Mech.},
   volume={467},
   date={2002},
   pages={289--306},
 
}


\bib{Doob42}{article}{
  title={The Brownian movement and stochastic equations},
  author={Doob, J. L.},
  journal={Annals of Mathematics},
  pages={351--369},
  year={1942},
  publisher={JSTOR}
}


\bib{DR00}{article}{
   author={J. Duchon},
   author={R. Robert},
   title={Inertial energy dissipation for weak solutions of incompressible Euler and Navier–Stokes equations},
   journal={Nonlinearity},
   volume={13},
   date={2000},
   pages={249--255},
 
}






\bib{E13}{book}{
   author={L.  C. Evans},
   title={An Introduction to Stochastic Differential Equations},
   publisher={American Mathematical Society},
   pages={151},
   date={2013},
 
}

\bib{FG95}{article}{
   author={Flandoli, F.},
   author={G\polhk atarek, D.},

  TITLE = {Martingale and stationary solutions for stochastic
              Navier-Stokes equations},
   JOURNAL = {Probab. Theory Related Fields},
  FJOURNAL = {Probability Theory and Related Fields},
    VOLUME = {102},
      YEAR = {1995},
    NUMBER = {3},
     PAGES = {367--391},
  
}

\bib{FSQD19}{article}{
   author={K. Fellner},
   author={S. Sonner},
   author={B. Q. Tang},
   author={D. D. Thuan},
   title={Stabilisation by noise on the boundary for a Chafee-Infante equation with dynamical boundary conditions},
   journal={AIMS},
   volume={24},
  date={2019},
   pages={4055-4078}
  
}



\bib{FJMRT}{article}{
    AUTHOR = {Foias, C.},
 AUTHOR = {Jolly, M. S.},
  AUTHOR = {Manley, O. P.},
   AUTHOR = {Rosa, R.},
    AUTHOR = {Temam, R.},
       TITLE = {Kolmogorov theory via finite-time averages},
   JOURNAL = {Phys. D},
  FJOURNAL = {Physica D. Nonlinear Phenomena},
    VOLUME = {212},
      YEAR = {2005},
    NUMBER = {3-4},
     PAGES = {245--270},
      ISSN = {0167-2789},
   MRCLASS = {76F02 (35Q30 76D05)},
  MRNUMBER = {2187512},
MRREVIEWER = {Xiaoming Wang},
       DOI = {10.1016/j.physd.2005.10.002},
       URL = {https://doi-org.proxyiub.uits.iu.edu/10.1016/j.physd.2005.10.002},
}


\bib{F95}{book}{
   author={U. Frisch},
   title={Turbulence},
   note={The legacy of A. N. Kolmogorov},
   publisher={Cambridge University Press, Cambridge},
   date={1995},
   pages={xiv+296},

}
\bib{H55}{article}{
   author={E. Hopf},
   title={Lecture series of the symposium on partial differential equations},
  journal={Berkeley},
  date={1955},
 }




\bib{H72}{article}{
   author={L. N. Howard},
   title={Bounds on flow quantities},
  
   journal={Annual Review of Fluid Mechanics},
   volume={4},
   date={1972},
 
   pages={473--494},
   
}

\bib{JL16}{article}{
   author={N. Jiang},
      author={W. J. Layton},
   title={Algorithms and models for turbulence not at statistical
equilibrium},
   journal={Computers and Mathematics with Applications},
   volume={71},
   date={2016},
   pages={2352--2372},
  
}




\bib{K97}{article}{
   author={R. R. Kerswell},
   title={Variational bounds on shear-driven turbulence and
turbulent Boussinesq convection },
   journal={Physica D},
   volume={100},
   date={1997},
   pages={355--376},
  
}

\bib{K99}{article}{
   author={A.A. Kwiecinska},
   title={Stabilization of partial differential equations by noise},
   journal={Stoch. Proc. $\&$ Appl.},
   volume={79},
   date={1999},
   pages={179–184},
  
}

\bib{K41}{article}{
   author={A. N. Kolmogorov},
   title={The local structure of turbulence in incompressible viscous fluid
   for very large Reynolds numbers},
   note={Translated from the Russian by V. Levin;
   Turbulence and stochastic processes: Kolmogorov's ideas 50 years on},
   journal={Proc. Roy. Soc. London Ser. A},
   volume={434},
   date={1991},
   number={1890},
   pages={9--13},
  
}




\bib{L16}{article}{
   author={Layton, W. J. },
   title={Energy dissipation in the Smagorinsky model of turbulence},
   journal={Appl. Math. Lett.},
   volume={59},
   date={2016},
   pages={56--59},

}
  


 
   \bib{L02}{article}{
   author={W. J. Layton },
   title={Energy dissipation bounds for shear flows for a model in large
   eddy simulation},
   journal={Math. Comput. Modelling},
   volume={35},
   date={2002},
   number={13},
   pages={1445--1451},
 
}


 
   \bib{LS18}{article}{
   author={Leslie, T. M.},
    author={Shvydkoy, R. },
   TITLE = {Conditions implying energy equality for weak solutions of the
              {N}avier-{S}tokes equations},
   JOURNAL = {SIAM J. Math. Anal.},
  FJOURNAL = {SIAM Journal on Mathematical Analysis},
    VOLUME = {50},
      YEAR = {2018},
    NUMBER = {1},
     PAGES = {870--890},
 
}



 \bib{M94}{article}{
   author={C. Marchioro},
   title={Remark on the energy dissipation in shear driven turbulence},
   journal={Phys. D},
   volume={74},
   date={1994},
   number={3-4},
   pages={395--398},
 
}


 \bib{MR04}{article}{
   author={Mikulevicius, R.},
      author={Rozovskii, B.L.},

   title={Stochastic Navier–Stokes equations for turbulent flows},
   journal={SIAM J. Math. Anal.},
   volume={35},
   issue={5},
   date={2004},

   pages={1250–1310},
  
}

 \bib{MR05}{article}{
   author={Mikulevicius, R.},
      author={Rozovskii, B.L.},

 TITLE = {Global {$L_2$}-solutions of stochastic {N}avier-{S}tokes
              equations},
   JOURNAL = {Ann. Probab.},
  FJOURNAL = {The Annals of Probability},
    VOLUME = {33},
      YEAR = {2005},
    NUMBER = {1},
     PAGES = {137--176},
  
}




\bib{AP17}{article}{
   author={A. Pakzad},
   title={Damping functions correct over-dissipation of the Smagorinsky Model},
  journal={ Mathematical Methods in the Applied Sciences},
   
volume = {40},
number = {16},
issn = {1099-1476},
pages = {5933--5945},
year = {2017},  
  
}

  
\bib{AP19}{article}{
   author={Pakzad, A.},
   title={Analysis of mesh effects on turbulence statistics},
  journal={Journal of Mathematical Analysis and Applications},
  volume = {475},
  pages = {839-860},
    year = {2019},  
  
}

\bib{AP20}{article}{
   AUTHOR = {A. Pakzad,},
     TITLE = {On the long time behavior of time relaxation model of fluids},
   JOURNAL = {Phys. D},
  FJOURNAL = {Physica D. Nonlinear Phenomena},
    VOLUME = {408},
      YEAR = {2020},
     PAGES = {132509, 6},
    
}


\bib{P00}{book}{
   author={Pope, S. B.},
   title={Turbulent flows},
   publisher={Cambridge University Press, Cambridge},
   date={2000},
   pages={xxxiv+771},
  
}


\bib{S77}{article}{
   author={V. Scheffer},
   title={Hausdorff measure and the Navier–Stokes equations},
   journal={ Comm. Math. Phys.},
   volume={55},
   date={1977},
   pages={97--112},
 
}




\bib{S98}{article}{
   author={Sreenivasan, K. R.},
   title={An update on the energy dissipation rate in isotropic turbulence},
   journal={Phys. Fluids},
   volume={10},
   date={1998},
   number={2},
   pages={528--529},
 
}



\bib{V14}{article}{
   author={J. C. Vassilicos},
   title={Dissipation in turbulent flows},
   journal={Annual Review of Fluid Mechanics},
   volume={47},
   date={2014},
  pages={95-114},
 
}


\bib{WW15}{article}{
   author={D. Wang},
      author={H.  Wang}
 TITLE = {Global existence of martingale solutions to the three-dimensional stochastic compressible Navier-Stokes  equations},
   JOURNAL = {Differential Integral Equations},
    FJOURNAL = {Differential and Integral Equations. An International Journal for Theory \& Applications},

    VOLUME = {28},
      YEAR = {2015},
    NUMBER = {11-12},
     PAGES = {1105--1154},
  
}



\bib{W10}{article}{
   author={X. Wang},
   title={Approximation of stationary statistical properties of dissipative dynamical systems: time discretization},
   journal={Mathematics of Computation},
   volume={79},
   date={2010},
   number={269},
   pages={259-280},
  
}

\bib{W00}{article}{
   author={Wang, X.},
   title={Effect of tangential derivative in the boundary layer on time averaged energy dissipation rate},
   journal={Physica D: Nonlinear Phenomena},
   volume={144},
   date={2000},
   issue={1-2},
   pages={142--153},
 
}


\bib{W97}{article}{
   author={X. Wang},
   title={Time-averaged energy dissipation rate for shear driven flows in
   $\bold R^n$},
   journal={Phys. D},
   volume={99},
   date={1997},
   number={4},
   pages={555--563},
  
}
  






\end{thebibliography}
\end{document}